\documentclass{amsart}

\RequirePackage{fix-cm}
\usepackage{graphicx}
\usepackage{amsmath, amsopn,amstext,amscd,amsfonts,amssymb}
\usepackage{dsfont}
\usepackage{comment}
\usepackage[active]{srcltx}
\usepackage{graphicx, epsfig, subfig}

\usepackage{textcomp}

\usepackage{algorithm}

\makeatletter
\def\BState{\State\hskip-\ALG@thistlm}
\makeatother

\def\downbar#1{
\setbox10=\hbox{$#1$}
            \dimen10=\ht10 \advance\dimen10 by 2.5pt
            \ifdim \dimen10<15pt 
               \advance\dimen10 by -0.5pt
               \dimen11=\dimen10
               \advance\dimen10 by 2.5pt
               \lower \dimen11
            \else \lower \ht10 \fi
            \hbox {\hskip 1.5pt \vrule height \dimen10 depth \dp10}}
\def\upbar#1{
\setbox10=\hbox{$#1$}
            \dimen10=\ht10 \advance\dimen10 by \dp10 \advance\dimen10 by 2.5pt
            \ifdim \dimen10<15pt 
                \advance\dimen10 by 2pt \fi
            \raise 2.5pt \hbox {\hskip -1.5pt \vrule height \dimen10}}

\usepackage{multicol}
\usepackage{colortbl}
\usepackage{exscale}
\usepackage{amssymb,latexsym,amsthm,amsfonts,color,fancyhdr,mathrsfs}
\usepackage{lscape}
\usepackage{rotating}
\usepackage{caption}

\newtheorem{definition}{\bf Definition}[section]
\newtheorem{theorem}{\bf Theorem}[section]
\newtheorem{proposition}{\bf Proposition}[section]
\newtheorem{lemma}{\bf Lemma}[section]

\newtheorem{remark}{\bf Remark}[section]

\numberwithin{equation}{section}
\begin{document}
\title[Coherent pair of measures on lattices]{Coherent pair of measures for orthogonal polynomials on lattices}

\author{D. Mbouna}
\address{Department of Mathematics, Faculty of Sciences, University of Porto, Campo Alegre st., 687, 4169-007 Porto, Portugal}
\email{dieudonne.mbouna@fc.up.pt}



\date{\today}
\maketitle
\small{\textit{This work is dedicated to the memory of my beloved friend and mentor J. Petronilho. To him, my eternal gratitude.}}
\\
\begin{abstract}
We consider two sequences of orthogonal polynomials $(P_n)_{n\geq 0}$ and $(Q_n)_{n\geq 0}$ with respect regular functionals ${\bf u}$ and ${\bf v}$, respectively. We assume that  
\begin{align*}
\sum_{j=1} ^{M} a_{j,n}\mathrm{D}_x ^k P_{k+n-j} (z)=\sum_{j=1} ^{N} b_{j,n}\mathrm{D}_x ^{m} Q_{m+n-j} (z)\;,
\end{align*}
with $k,m,M,N \in \mathbb{N}$, $a_{j,n}$ and $b_{j,n}$ are sequences of complex numbers, $$2\mathrm{S}_xf(x(s))=(\triangle +2\,\mathrm{I})f(z),~~ \mathrm{D}_xf(x(s))=\frac{\triangle}{\triangle x(s-1/2)}f(z),$$ $z=x(s-1/2)$, $\mathrm{I}$ is the identity operator, $x$ defines a lattice, and $\triangle f(s)=f(s+1)-f(s)$. We show that under some natural conditions, the functionals ${\bf u}$ and ${\bf v}$ are connected by a rational factor whenever $m=k$, and for 
$k>m$, ${\bf u}$ and ${\bf S}_x ^{k-m}{\bf v}$ are semiclassical functionals and in addition ${\bf S}_x{\bf u}$ and ${\bf S}_x ^{k-m+1}{\bf v}$ are connected by a rational factor. 
This leads to the notion of $(M,N)$-coherent pair of measures of order $(m,k)$ extended to orthogonal polynomials on lattices.

%
\end{abstract}

\section{Introduction}\label{introduction}
Many interesting problems in the theory of orthogonal polynomial sequences (OPS) are inverse problems. That is to analyse properties of linear functionals from the fact that their corresponding OPS satisfy an algebraic equation. For instance, the concept of coherent pair of measures was introduced by Iserles et al. \cite{IKNS1991} motivated by application to certain Sobolev inner product. In \cite{JMPN2015, JP2013} this notion was extended to the notion of $(M,N)-$coherent pair of order $(m,k)$ that is described as follows. Given two monic OPS $(P_n)_{n\geq0}$ and $(Q_n)_{n\geq0}$, we say that $\big((P_n)_{n\geq0},(Q_n)_{n\geq0}\big)$ is an $(M,N)-$coherent pair of order $(m,k)$ if there exist two non-negative integer numbers $M$ and $N$, and sequences of complex numbers
$(a_{n,j})_{n\geq0}$ ($j=0,1,\ldots,M$) and $(b_{n,j})_{n\geq0}$ ($j=0,1,\ldots,N$) such that,
under natural assumptions on the coefficients $a_{n,j}$ and $b_{n,j}$, the structure relation
\begin{align}\label{coherent-pair-intro}
\sum_{j=0}^{M} a_{n,j} P_{n-j}^{[m]} (x) = \sum_{j=0}^{N} b_{n,j} Q_{n-j}^{[k]} (x) \quad (n=0,1,\ldots)
\end{align}
holds. Here we denote
$$
P_{n}^{[m]}(x):=r_n D^m P_{n+m}(x)
$$
by the normalized sequence of ``derivative" of $P_n$ of order $m$ with respect to the operator $D=d/dx$. In \cite{RJPS2015}, this notion is extended for the operators $D=D_q$ and $D=\triangle_{\omega}$, where $$D_qf(x)=\frac{f(qx)-f(x)}{(q-1)x},~~\quad \triangle_{\omega}f(x)=\frac{f(x+\omega)-f(x)}{\omega},~~q\neq 1,~\omega \in \mathbb{C}\setminus \left\lbrace 0 \right\rbrace \;.$$

Recently in \cite{RKDP2020, KCDM2020}, the notion of $\pi_N$-coherent pair with index $M$ of order $(m,k)$ is introduced as a generalization of the concept of $\pi_N$-coherent pair with index $N$ introduced by Maroni and Sfaxi in \cite{MS2000}. This is the situation of \eqref{coherent-pair-intro} where the left hand side is replaced by $\pi_NP_n ^{[m]}$, i.e.
\begin{align}\label{coherent-pair-intro-b}
\pi_N(x)P_{n}^{[m]} (x) = \sum_{j=n-M}^{N} c_{n,j} Q_{j}^{[k]} (x)\;, \quad c_{n,n-M}\neq 0\;,~~(n=0,1,\ldots)
\end{align}
where $\pi_N$ is a polynomial of degree $N$. For a complete history about \eqref{coherent-pair-intro-b} we refer the reader to \cite{RKDP2020, BLN1987, KCDM2020, MS2000} and references therein. 

Let $\textbf{u}$ and $\textbf{v}$ be the moment regular functionals with respect to which
$(P_n)_{n\geq0}$ and $(Q_n)_{n\geq0}$ are respectively orthogonal. Under some conditions imposed only on the coefficients of the considered relation, it is remarkable that either in the situation where \eqref{coherent-pair-intro} or \eqref{coherent-pair-intro-b} holds, and for any of the above mentioned operators $d/dx$, $D_q$ and $\triangle_{\omega}$, functionals $\textbf{u}$ and $\textbf{v}$ are connected by a rational transformation (in the distributional sense) whenever $m=k$, i.e., there exist nonzero polynomials $\Phi$ and $\Psi$ such that $$\Phi\textbf{u}=\Psi\textbf{v}\;.$$ When $m \neq k$, $\textbf{u}$ and $\textbf{v}$ are still connected by a rational transformation and, in addition, they are semiclassical. This means there exist four nonzero polynomials $\Phi_1$, $\Phi_2$,$\Psi_1$ and $\Psi_2$ such that $${\bf D}(\Phi_1 {\bf u})=\Psi_1 {\bf u} ,~~ {\bf D}(\Phi_2 {\bf v})=\Psi_2 {\bf v} \;.$$ 

 The main objective of this work is to extend this notion to the theory of OPS on lattices and therefore generalize many known results in the literature about the subject (see \cite{KLM2001, MBP1993, MBP1994, MMM2001} and references therein). For this reason we consider \eqref{coherent-pair-intro} (and \eqref{coherent-pair-intro-b} for $k=0$) for an operator on lattices defined by 
$$D_xf(x(s))=\frac{\triangle}{\triangle x(s-1/2)} f(x(s-1/2))\;,$$ where $x(s)$ is a lattice and $\triangle f(s)=f(s+1)-f(s)$ and study the semiclassical character of the involved OPS. This will be done thanks to the recent result presented in \cite{KDP2020a} and contrary to the previously mentioned cases, results for this operator are unexpected (see Remark \ref{clarifications}).

 The structure of this work is as follows. Section 2 presents some basic facts of the algebraic theory of OPS on lattices. In Section \ref{main-results} our main results are stated and proved.

\section{Background and Preliminary}
Along this note we are going to use the algebraic approach on OPS developed by P. Maroni \cite{M1991}. Here are some basic facts of the theory. 

 Let $\mathcal{P}$ be the vector space of all polynomials with complex coefficients
and let $\mathcal{P}^*$ be its algebraic dual. A simple set in $\mathcal{P}$ is a sequence $(R_n)_{n\geq0}$ such that $\mathrm{deg}(R_n)=n$ for each $n$. A simple set $(R_n)_{n\geq0}$ is called an OPS with respect to ${\bf w}\in\mathcal{P}^*$ if 
$$
\langle{\bf w},R_nR_m\rangle=h_n\delta_{n,m}\;,\quad m=0,1,\ldots;\;h_n\in\mathbb{C}\setminus\{0\}\;.
$$
In this case, we say that ${\bf w}$ is  regular. The left multiplication of a functional ${\bf w}$ by a polynomial $\phi$ is defined by
$$
\left\langle \phi {\bf w}, p  \right\rangle =\left\langle {\bf w},\phi p  \right\rangle\;, \quad p\in \mathcal{P}\;.
$$

A dual basis $({\bf r}_n)_{n\geq 0}$ of a simple set polynomial sequence $(R_n)_{n\geq 0}$ is a sequence in $\mathcal{P}^*$ such that $\left\langle {\bf r}_n, R_m   \right\rangle =\delta_{n,m}$, for all $n,m$. In addition any functional ${\bf v} \in \mathcal{P}^*$ (when $\mathcal{P}$ is endowed with an appropriate strict inductive limit topology, see \cite{M1991}) can be written in the sense of the weak topology in $\mathcal{P}^*$ as 
\begin{align*}
{\bf v} = \sum_{n=0} ^{\infty} \left\langle {\bf v}, R_n \right\rangle {\bf r}_n\;,
\end{align*}
where $({\bf r}_n)_{n\geq 0}$ is the dual basis associated to the sequence of simple set polynomials $(R_n)_{n\geq 0}$. Consequently, if $(R_n)_{n\geq0}$ is a (monic) OPS with respect to ${\bf w}\in\mathcal{P}^*$, then the corresponding dual basis is explicitly given by 
\begin{align}\label{expression-an}
{\bf r}_n =\left\langle {\bf w} , R_n ^2 \right\rangle ^{-1} R_n{\bf w}\;.
\end{align}

It is known (see \cite{C1978}) that a monic OPS, $(R_n)_{n\geq 0}$, is characterized by the following three-term recurrence relation (TTRR):
\begin{align}\label{TTRR_relation}
R_{-1} (z)=0, \quad R_{n+1} (z) =(z-B_n)R_n (z)-C_n R_{n-1} (z)\;, \quad C_n \neq 0\;,
\end{align}
and, therefore,
\begin{align}\label{TTRR_coefficients}
B_n = \frac{\left\langle {\bf w} , zR_n ^2 \right\rangle}{\left\langle {\bf w} , R_n ^2 \right\rangle},\quad C_{n+1}  = \frac{\left\langle {\bf w} , R_{n+1} ^2 \right\rangle}{\left\langle {\bf w} , R_n ^2 \right\rangle}.
\end{align}

In our framework, a lattice $x$ is a mapping given by (see \cite{ARS1995})
\begin{equation}
\label{xs-def}
x(s):=\left\{
\begin{array}{lcl}
\mathfrak{c}_1 q^{-s} +\mathfrak{c}_2 q^s +\mathfrak{c}_3,&  q\neq1\\ [7pt]
\mathfrak{c}_4 s^2 + \mathfrak{c}_5 s +\mathfrak{c}_6, &  q =1,
\end{array}
\right.
\end{equation}
where $q>0$ and $\mathfrak{c}_j$ ($1\leq j\leq6$) are complex numbers such that $(\mathfrak{c}_1,\mathfrak{c}_2)\neq(0,0)$ if $q\neq1$.
Note that
$x\big(s+\frac12\big)+x\big(s-\frac12\big)=2\alpha x(s)+2\beta,$
where
\begin{equation}\label{alpha-beta}
\alpha=\frac{q^{1/2}+q^{-1/2}}{2},\quad
\beta=\left\{
\begin{array}{lcl}
(1-\alpha)\mathfrak{c}_3, &  q\neq1,\\ [7pt]
\mathfrak{c}_4/4, &  q =1.
\end{array}
\right.
\end{equation}
We define $\alpha_n:=(q^{n/2} +q^{-n/2})/2$ and
\begin{align*}
\gamma_n := \left\{
\begin{array}{lcl}
\displaystyle\frac{q^{n/2}-q^{-n/2}}{q^{1/2}-q^{-1/2}}, & q\neq1 \\ [7pt]
n, & q=1.
\end{array}
\right. 
\end{align*}
We set $\gamma_{-1}:=-1$ and $\alpha_{-1}:=\alpha$. 
We also define two operators $\mathrm{D}_x$ and $\mathrm{S}_x$ on $\mathcal{P}$ by 
\begin{align*}
\mathrm{D}_x f(x(s))&=\frac{\triangle}{{ \triangle} x(s-1/2)}f(x(s-1/2))\;,\\
\mathrm{S}_x f(x(s))&= \frac{1}{2}(\triangle+2\,\mathrm{I})f(x(s-1/2))\;.
\end{align*}
These operators  induce two elements on $\mathcal{P}^*$, say $\mathbf{D}_x$ and $\mathbf{S}_x$, via the following definition (see \cite{FK-NM2011}): 
\begin{align*}
\langle \mathbf{D}_x{\bf u},f\rangle=-\langle {\bf u},\mathrm{D}_x f\rangle,\quad \langle\mathbf{S}_x{\bf u},f\rangle=\langle {\bf u},\mathrm{S}_x f\rangle.
\end{align*}

\begin{definition} \cite{FK-NM2011}
A regular functional ${\bf u}$ is said to be semiclassical if there exist two nonzero polynomials $\phi$ and $\psi$ such that the following distributional equation $$\phi{\bf D}_x{\bf u} =\psi {\bf S}_x{\bf u}$$ holds. We also say that the corresponding OPS is semiclassical. If polynomials $\phi$ and $\psi$ are of degrees at most two and one, respectively, then ${\bf u}$ is said classical functional. 
\end{definition}

 Let $f,g\in\mathcal{P}$ and ${\bf u}\in\mathcal{P}^*$. Then the following properties hold (see e.g. \cite{KDP2020a, FK-NM2011, SMMFPN2017}):
\begin{align}
\mathrm{D}_x \big(fg\big)&= \big(\mathrm{D}_x f\big)\big(\mathrm{S}_x g\big)+\big(\mathrm{S}_x f\big)\big(\mathrm{D}_x g\big), \label{def-Dx-fg} \\[7pt]
\mathrm{S}_x \big( fg\big)&=\big(\mathrm{D}_x f\big) \big(\mathrm{D}_x g\big)\texttt{U}_2  +\big(\mathrm{S}_x f\big) \big(\mathrm{S}_x g\big), \label{def-Sx-fg} \\[7pt]
f\mathrm{D}_xg&=\mathrm{D}_x\left[ \Big(\mathrm{S}_xf-\frac{\texttt{U}_1}{\alpha}\mathrm{D}_xf \Big)g\right]-\alpha ^{-1}\mathrm{S}_x \Big(g\mathrm{D}_x f\Big) , \label{def-fDxg} \\[7pt]
{\bf D}_x (f{\bf u})&=\Big(\mathrm{S}_xf-\alpha ^{-1} \texttt{U}_1\mathrm{D}_xf \Big){\bf D}_x {\bf u}+\alpha ^{-1}\mathrm{D}_xf~{\bf S}_x{\bf u},\label{def_D_xfu}\\[7pt]
{\bf S}_x (f{\bf u})&=\Big(\alpha \texttt{U}_2 -\alpha^{-1}\texttt{U}_1 ^2 \Big)\mathrm{D}_x f~{\bf D}_x{\bf u} +\Big(\mathrm{S}_xf+\alpha ^{-1} \texttt{U}_1\mathrm{D}_xf \Big){\bf S}_x{\bf u},\label{def_S_xfu}\\[7pt]
f{\bf D}_x {\bf u}&={\bf D}_x\left(S_xf~{\bf u}  \right)-{\bf S}_x\left(D_xf~{\bf u}  \right), \label{def-fD_x-u}\\[7pt]
\alpha \mathbf{D}_x ^n \mathbf{S}_x {\bf u}&= \alpha_{n+1} \mathbf{S}_x \mathbf{D}_x^n {\bf u}
+\gamma_n \texttt{U}_1\mathbf{D}_x^{n+1}{\bf u}, \label{DxnSx-u} 
\end{align}
where
\begin{align*}
\texttt{U}_1 (z)&=\left\{
\begin{array}{lcl}
(\alpha^2-1)\big(z-\mathfrak{c}_3\big), & q\neq1\\[7pt]
2\beta, &q=1,
\end{array}
\right. \\[7pt]
\texttt{U}_2(z)&=\left\{
\begin{array}{lcl}
(\alpha^2-1)\big((z-\mathfrak{c}_3)^2-4\mathfrak{c}_1\mathfrak{c}_2\big), & q\neq1\\[7pt]
4\beta (z-\mathfrak{c}_6)+\mathfrak{c}_5^2/4, & q=1.
\end{array}
\right.
\end{align*}

We denote by $R_n ^{[k]}$ $(k=0,1,\ldots)$ the monic polynomial of degree $n$ defined by
\begin{align*}
R_n ^{[k]} (z)=\frac{\gamma_{n} !}{\gamma_{n+k} !} \mathrm{D}_x ^k R_{n+k} (z)\;,
\end{align*}
with $\gamma_0 !=1$, $\gamma_{n+1}!=\gamma_1\cdots \gamma_n \gamma_{n+1}$. If $({\bf r}^{[k]} _n)_{n\geq 0}$ is the dual basis associated to the sequence $(R_n ^{[k]})_{n\geq 0}$, it is known that
\begin{align}
{\bf D}_x ^k {\bf r}^{[k]} _n=(-1)^k \frac{\gamma_{n+k}!}{\gamma_n ! }{\bf r}_{n+k}\;,\quad k=0,1,\ldots\;. \label{basis-Dx-derivatives}
\end{align}
The following result is helpful.
\begin{proposition}\cite{KDP2020a} \label{Leibniz-rule-NUL}
Let ${\bf u}\in\mathcal{P}^*$ and $f\in\mathcal{P}$. Then
\begin{align}
\mathbf{D}_x ^n \big(f{\bf u}\big)
=\sum_{k=0}^{n} \mathrm{T}_{n,k}f\, \mathbf{D}_x^{n-k} \mathbf{S}_x^k {\bf u}
 \quad (n=0,1,2,\ldots),\label{leibnizfor-NUL}
\end{align}
where $\mathrm{T}_{n,k}f$ is a polynomial defined by $\mathrm{T}_{0,0}f =f$ and
\begin{align}\label{Tnk}
\mathrm{T}_{n,k}f&:= \mathrm{S}_x \mathrm{T}_{n-1,k}f
-\frac{\gamma_{n-k}}{ \alpha_{n-k}}\texttt{U}_1 \mathrm{D}_x \mathrm{T}_{n-1,k}f
+\frac{1}{\alpha_{n+1-k}} \mathrm{D}_x \mathrm{T}_{n-1,k-1}f\,,
\end{align}
with $\mathrm{T}_{n,k}f=0$ whenever $k>n$ or $k<0$. Moreover 
$\deg \mathrm{T}_{n,k}f\leq\deg f-k$.
\end{proposition}

\begin{definition}
Two functionals ${\bf u}$ and ${\bf v}$ are connected by a rational modification (in the distribution sense) if there exist two nonzero polynomials $\phi$ and $\psi$ such that $\phi {\bf u}=\psi {\bf v}$.
\end{definition}

\section{main results}\label{main-results}
The following fact is known for the following operators $\left\lbrace \frac{d}{dx}, \triangle_{\omega}, D_q \right\rbrace$, where $D_q$ is the $q$-Jackson operator, but for the $D_x$ operator, as far as we know, this is new.
\begin{proposition}
Let ${\bf u}$ and ${\bf v}$ be two functionals connected by a rational modification. Then if one of them is semiclassical, then so is the other one.
\end{proposition} 
\begin{proof}
Let ${\bf u}$ and ${\bf v}$ be two functionals and assume that ${\bf u}$ is semiclassical. Then
\begin{align}\label{relation-rational-modification}
\pi_2 {\bf u}=\pi_1 {\bf v}\;,\quad \phi {\bf D}_x {\bf u}=\psi {\bf S}_x{\bf u}\;,,
\end{align}
for some nonzero polynomials $\phi$, $\psi$, $\pi_1$ and $\pi_2$. \\
We firstly apply ${\bf D}_x$ to the first relation in \eqref{relation-rational-modification} using \eqref{def_D_xfu} to obtain
\begin{align*}
\big( \mathrm{S}_x \pi_2 -\frac{\texttt{U}_1}{\alpha}\mathrm{D}_x\pi_2 \big){\bf D}_x{\bf u} +\frac{1}{\alpha}\mathrm{D}_x\pi_2 {\bf S}_x{\bf u}=\big( \mathrm{S}_x \pi_1 -\frac{\texttt{U}_1}{\alpha}\mathrm{D}_x\pi_1 \big){\bf D}_x{\bf v} +\frac{1}{\alpha}\mathrm{D}_x\pi_1 {\bf S}_x{\bf v}\;.
\end{align*}
Hence 
\begin{align}\label{with-K1}
K_1{\bf S}_x{\bf u}= \phi \Big(\big( \mathrm{S}_x \pi_1 -\frac{\texttt{U}_1}{\alpha}\mathrm{D}_x\pi_1 \big){\bf D}_x{\bf v} +\frac{1}{\alpha}\mathrm{D}_x\pi_1 {\bf S}_x{\bf v}  \Big)\;,
\end{align}
where $\alpha K_1=\phi \mathrm{D}_x \pi_2 +(\alpha \mathrm{S}_x \pi_2 -\mathtt{U}_1 \mathrm{D}_x \pi_2)\psi$, holds by multiplying the above equation by $\phi$ and using the second equation in \eqref{relation-rational-modification}. \\
We secondly apply ${\bf S}_x$ to the first relation in \eqref{relation-rational-modification} using \eqref{def_S_xfu} and repeat the process to obtain
\begin{align}\label{with-K2}
K_2{\bf S}_x{\bf u}= \phi \Big(\big(\alpha \texttt{U}_2 -\frac{\texttt{U}_1 ^2}{\alpha}  \big)\mathrm{D}_x\pi_1{\bf D}_x{\bf v} +\big( \mathrm{S}_x \pi_1 +\frac{\texttt{U}_1}{\alpha}\mathrm{D}_x\pi_1 \big){\bf S}_x{\bf v}   \Big)\;,
\end{align}
where $\alpha K_2= \phi \big(\alpha \mathrm{S}_x \pi_2 +\texttt{U}_1\mathrm{D}_x\pi_2 \big)  +\psi \big(\alpha^2 \texttt{U}_2-\texttt{U}_1 ^2  \big)\mathrm{D}_x\pi_2$.
By eliminating ${\bf S}_x{\bf u}$ in \eqref{with-K1}--\eqref{with-K2}, we finally obtain
$$\Phi {\bf D}_x{\bf v}+\Psi {\bf S}_x{\bf v}=0\;,$$
where
\begin{align*}
\Phi&=\Big(\phi \mathrm{D}_x \pi_2 +(\alpha \mathrm{S}_x \pi_2 -\mathtt{U}_1 \mathrm{D}_x \pi_2)\psi  \Big)\big(\texttt{U}_1 ^2 -\alpha ^2\texttt{U}_2  \big)\mathrm{D}_x\pi_1 \\
&+\big(\alpha\mathrm{S}_x \pi_1-\mathtt{U}_1\mathrm{D}_x\pi_1\big)\Big(\phi \big(\alpha \mathrm{S}_x \pi_2 +\texttt{U}_1\mathrm{D}_x\pi_2 \big)  +\psi \big(\alpha^2 \texttt{U}_2-\texttt{U}_1 ^2  \big)\mathrm{D}_x\pi_2  \Big)\;,  \\
\Psi&=\big(\alpha \mathrm{S}_x\pi_1 +\texttt{U}_1\mathrm{D}_x\pi_1  \big) \Big(\phi \mathrm{D}_x \pi_2 +(\alpha \mathrm{S}_x \pi_2 -\mathtt{U}_1 \mathrm{D}_x \pi_2)\psi \big)  \\
&-\Big(\phi \big(\alpha \mathrm{S}_x \pi_2 +\texttt{U}_1\mathrm{D}_x\pi_2 \big)  +\psi \big(\alpha^2 \texttt{U}_2-\texttt{U}_1 ^2  \big)\mathrm{D}_x\pi_2  \Big)\mathrm{D}_x\pi_1,\\
\end{align*}
and the desired result follows.
\end{proof}

We start with the following lemma using ideas developed in \cite{RJPS2015, JP2006}.
\begin{lemma}\label{lemma1}
Let $({\bf u}, {\bf v})$ be a pair of regular functionals and $((P_n)_{n\geq 0}, (Q_n)_{n\geq 0})$ the corresponding pair of monic OPS. Assume that for some $k,m,M,N \in \mathbb{N}$, we have 
\begin{align}\label{general_problem}
\sum_{j=0} ^{M} a_{j,n} P_{n-j} ^{[k]} (z)=\sum_{j=0} ^{N} b_{j,n} Q_{n-j} ^{[m]} (z)\;, \quad \quad n=0,1,\ldots\;,
\end{align}
for some complex sequences $a_{i,n}$ and $b_{i,n}$, with $a_{0,n}=1=b_{0,n}$ and $a_{M,n}b_{N,n}\neq 0$ for all $n$. Let $\mathcal{A}_{M+N}=\Big[l_{i,j}\Big]_{i,j=0} ^{M+N-1}$ be the following matrix of order $M+N$,
\begin{align*}
l_{i,j} = \left\{
\begin{array}{lclcl}
\displaystyle a_{j-i,j}, ~ \textit{if}~~0\leq i\leq N-1~\textit{and}~ i\leq j\leq M+i  \\ [7pt]
b_{j-i+N,j}, ~ \textit{if}~~N\leq i\leq M+N-1~\textit{and}~ i-N\leq j\leq i \\
0,~ \textit{else}
\end{array}
\right. 
\end{align*}
Assume that $det(\mathcal{A}_{M+N})\neq 0$ and $k\geq m$. Then there exist two polynomials $\psi_{N+k+n}$ and $\phi_{M+m+n}$ with degrees $N+k+n$ and $M+m+n$, respectively, such that 
\begin{align}\label{relation_with_dual_basis}
\psi_{N+k+n}{\bf u}= {\bf D}_x ^{k-m}\Big(\phi_{M+m+n}{\bf v} \Big)\;, ~\quad\quad~n=0,1,\ldots\;.
\end{align} 
\end{lemma}

\begin{proof}
Define 
$$R_n(z)= \sum_{j=0} ^{M} a_{j,n} P_{n-j} ^{[k]} (z)\;, ~~\quad n=0,1,\ldots\;.$$ Then $(R_n)_{n\geq 0}$ is a simple set of polynomials.
Let $({\bf a}_n)_{n\geq 0}$, $({\bf b}_n)_{n\geq 0}$, $({\bf a}_n ^{[k]})_{n\geq 0}$, $({\bf b}_n ^{[m]})_{n\geq 0}$ and $({\bf r}_n)_{n\geq 0}$ be the associated dual basis to the sequences $(P_n)_{n\geq 0}$, $(Q_n)_{n\geq 0}$, $(P_n ^{[k]})_{n\geq 0}$, $(Q_n ^{[m]})_{n\geq 0}$ and $(R_n)_{n\geq 0}$, respectively.
We are going to prove that 
\begin{align}
{\bf a}_n ^{[k]} = \sum_{l=n} ^{n+M} a_{l-n,l}{\bf r}_l \;,\quad {\bf b}_n ^{[m]} = \sum_{l=n} ^{n+N} b_{l-n,l}{\bf r}_l\;,  \quad n=0,1,\ldots\;.\label{expressions_ank_Sxbnm}
\end{align}
Indeed, by definition of $R_n$, we obtain
\begin{align*}
\left\langle {\bf a}_n ^{[k]},R_l  \right\rangle=\sum_{i=0} ^M a_{i,l}\left\langle {\bf a}_n ^{[k]},P_{l-i} ^{[k]}  \right\rangle =\left\{
    \begin{array}{ll}
        a_{l-n,l} & \mbox{if } n\leq l\leq n+M \\
        0 & \mbox{otherwise.}
    \end{array}
\right.
\end{align*}
Similarly, using \eqref{general_problem}, we write
\begin{align*}
\left\langle  {\bf b}_n ^{[m]},R_l  \right\rangle=\sum_{i=0} ^N b_{i,l}\left\langle {\bf b}_n ^{[m]},Q_{l-i} ^{[m]}  \right\rangle =\left\{
    \begin{array}{ll}
        b_{l-n,l} & \mbox{if } n\leq l\leq n+N \\
        0 & \mbox{otherwise.}
    \end{array}
\right.
\end{align*}
Therefore \eqref{expressions_ank_Sxbnm} hold by writing 
$${\bf a}_n ^{[k]}=\sum_{l=0} ^{\infty} \left\langle  {\bf a}_n ^{[k]},R_l  \right\rangle {\bf r}_l \;,  \quad {\bf b}_n ^{[m]}=\sum_{l=0} ^{\infty} \left\langle  {\bf b}_n ^{[m]},R_l  \right\rangle {\bf r}_l\;, $$
and by using what is preceding. Taking $n=0,1,\ldots,N-1$ and $n=0,1,\ldots, M-1$ in the first and in the second equation of \eqref{expressions_ank_Sxbnm}, respectively, we obtain a system of equations whose matrix is $\mathcal{A}_{M+N}$ and since $det(\mathcal{A}_{M+N})\neq 0$, then we may write 
$$r_n= \sum_{i=0} ^{N-1} \widehat{a}_{n,i}{\bf a}_i ^{[k]} +\sum_{j=0} ^{M-1}\widehat{b}_{n,j}{\bf b}_j ^{[m]}\;,\quad
n=0,1,\ldots, M+N-1,$$
for some complex sequences $\widehat{a}_{n,i}$ and $\widehat{b}_{n,j}$. Even more we may also write
\begin{align}
\sum_{i=0} ^{n+N} a'_{n,i}{\bf a}_i ^{[k]} =\sum_{j=0} ^{n+M}b'_{n,j}{\bf b}_j ^{[m]}\;,\quad n=0,1,\ldots\;.\label{equation_with_basis00}
\end{align}
with $a'_{n,i}$ and $b'_{n,j}$ complex sequences with $a'_{n,n+N}=b_{N,M+N+n}$ and $b'_{n,n+M}=a_{M,M+N+n}$. 
Now since $k\geq m$, we apply ${\bf D}_x ^k$ to \eqref{equation_with_basis00} using \eqref{basis-Dx-derivatives} to obtain
$$\sum_{i=0} ^{n+N}(-1)^k \frac{\gamma_{k+i}!}{\gamma_i !} a'_{n,i}{\bf a}_{k+i} ={\bf D}_x ^{k-m}\Big(\sum_{j=0} ^{n+M}(-1)^m \frac{\gamma_{m+j}!}{\gamma_j !}b'_{n,j} {\bf b}_{m+j} \Big)\;\quad (n=0,1,\ldots) $$
Hence \eqref{relation_with_dual_basis} holds where
\begin{align*}
\psi_{N+k+n}(z)&=\sum_{i=0} ^{n+N} \frac{(-1)^k\gamma_{k+i}!}{\gamma_i !\left\langle {\bf u},P_{k+i} ^2 \right\rangle}a'_{n,i}P_{k+i}(z)\;, \\
\phi_{M+m+n}(z)&=\sum_{j=0} ^{n+M} \frac{(-1)^m\gamma_{m+j}!}{\gamma_j! \left\langle {\bf v},Q_{m+j} ^2 \right\rangle}b'_{n,j}Q_{m+j}(z)\;. \\
\end{align*}
In addition, since $a'_{n,n+N}b'_{n,n+M}=b_{N,M+N+N}a_{M,M+N+n} \neq 0$, we clearly have 
$\deg \psi_{N+k+n}=N+k+n$, $\deg \phi_{M+m+n}=M+m+n$. Thus the desired result follows.
\end{proof}
Let us now state the first result.
\begin{theorem}\label{main-result-general-for-m=k}
Let $({\bf u}, {\bf v})$ be a pair of regular functionals with respect to the pair of monic OPS $\big((P_n)_{n\geq 0},(Q_n)_{n\geq 0}\big)$. Assume that \eqref{general_problem} holds with $k\geq m$. Under the assumptions and conclusion of Lemma \ref{lemma1}, assume further that $\det (\mathcal{B}_{m-k+4}) \neq 0$, where $\mathcal{B}_{k-m+3}(z)=\Big[c_{i,j}(z)\Big]_{i,j=0} ^{k-m+2}$ is the following polynomial matrix of order $k-m+3$
\begin{align*}
c_{i,j}(z) = \left\{
\begin{array}{lclcl}
\displaystyle \texttt{U}_1\mathrm{D}_x\psi_{N+k+i}(z)-\alpha \mathrm{S}_x\psi_{N+k+i}(z), ~ \textit{if}~~j=0  \\ [7pt]
\alpha \mathrm{T}_{k-m+1,j}\phi_{M+m+i}(z),~\textit{otherwise .}
\end{array}
\right. 
\end{align*}

Then the following hold.
\begin{itemize}
\item[i. ]If $m=k$, then ${\bf u}$ and ${\bf v}$ are connected by a rational modification and so ${\bf u}$ is semiclassical if and only if ${\bf v}$ is so.\\
\item[ii. ]If $k>m$, then  ${\bf u}$ and ${\bf S}_x ^{k-m}{\bf v}$ are semiclassical functionals and in addition ${\bf S}_x{\bf u}$ and ${\bf S}_x ^{k-m+1}{\bf v}$ are connected by a rational modification. That is, there exist six nonzero polynomials $\phi_1$, $\phi_2$, $\psi_1$ $\phi_3$, $\psi_2$ and $\psi_3$, such that
$$\phi_1{\bf D}_x{\bf u}=\psi_1{\bf S}_x{\bf u}\;,~ \phi_2{\bf S}_x{\bf u}=\psi_2{\bf S}_x ^{k-m+1}{\bf v}\;,~ \phi_3{\bf D}_x{\bf S}_x ^{k-m}{\bf v}=\psi_3{\bf S}_x ^{k-m+1} {\bf v}\;.$$
  
\end{itemize}
\end{theorem}

\begin{proof}
The case $k=m$ is trivial. Assume $k>m$. Now we apply ${\bf D}_x$ to \eqref{relation_with_dual_basis} using \eqref{def_D_xfu} and \eqref{leibnizfor-NUL} to obtain
\begin{align*}
\mathrm{D}_x\psi_{N+k+n}{\bf S}_x{\bf u}&=\Big(\texttt{U}_1\mathrm{D}_x\psi_{N+k+n}-\alpha \mathrm{S}_x\psi_{N+k+n} \Big) {\bf D}_x{\bf u} \\
&+\alpha \sum_{j=0} ^{k-m+1}\mathrm{T}_{k-m+1,j}\phi_{M+m+n} {\bf D}_x ^{k-m+1-j}{\bf S}_x ^j {\bf v} \;,
\end{align*}
for $n=0,1,\ldots$. Taking $n=0,1,2,\ldots, k-m+2$, we obtain the following system 
\begin{align*}
\begin{bmatrix}
\mathrm{D}_x\psi_{N+k}{\bf S}_x{\bf u}\\
\mathrm{D}_x\psi_{N+k+1}{\bf S}_x{\bf u}\\
\mathrm{D}_x\psi_{N+k+2}{\bf S}_x{\bf u}\\
.\\
.\\
.\\
\mathrm{D}_x\psi_{N+m+2}{\bf S}_x{\bf u}\\
\mathrm{D}_x\psi_{N+m+3}{\bf S}_x{\bf u}
\end{bmatrix} =\mathcal{B}_{k-m+3} \begin{bmatrix}
{\bf D}_x{\bf u}\\
{\bf D}_x ^{k-m+1}{\bf v}\\
{\bf D}_x ^{k-m}{\bf S}_x{\bf v}\\
.\\
.\\
.\\
{\bf D}_x{\bf S}_x ^{k-m}{\bf v}\\
{\bf S}_x ^{k-m+1}{\bf v}
\end{bmatrix}\;.
\end{align*}
Since $B(z)=\det (\mathcal{B}_{k-m+3}(z))\neq 0$, then we can solve this system for ${\bf D}_x{\bf u}$, ${\bf D}_x{\bf S}_x ^{k-m}{\bf v}$ and ${\bf S}_x ^{k-m+1}{\bf v}$. That is, there exist nonzero polynomials $\pi_1$, $\pi_2$ and $\pi_3$ such that 
\begin{align}
&B{\bf D}_x{\bf u}=\pi_1{\bf S}_x{\bf u}\;, \label{first-solu}\\
&B {\bf D}_x{\bf S}_x ^{k-m}{\bf v}=\pi_2{\bf S}_x{\bf u}\;, \label{second-solu}\\
&B {\bf S}_x ^{k-m+1}{\bf v}=\pi_3{\bf S}_x{\bf u}\;. \label{third-solu}
\end{align}
In addition from \eqref{second-solu} and \eqref{third-solu}, we obtain
\begin{align}
\pi_3 {\bf D}_x{\bf S}_x ^{k-m}{\bf v}=\pi_2 {\bf S}_x ^{k-m+1}{\bf v}\;.\label{from-secon-and-third-solu}
\end{align}
The result follows from \eqref{first-solu}, \eqref{second-solu} and \eqref{from-secon-and-third-solu}.
\end{proof}

\begin{remark}\label{clarifications}
It is important to notice that for a $(M,N)$-coherent pair of measures of order $(m,k)$ on nonuniform lattices with $k\neq m$, only one of the involved OPS is semiclassical and the semiclassical character of the other one cannot be insured. In addition, even the rational modification of such OPS cannot be insured. These make the difference with known results with operators $d/dx$, $\triangle_{\omega}$ and $D_q$ in \cite{RJPS2015, KLM2001, MBP1993, MMM2001, JMPN2015, JP2008, JP2013} and some references therein. 
\end{remark}

Remark \ref{clarifications} also applies to the notion of $\pi_N$-coherent pair of measures of order $(m,k)$ with index $M$. This is the case of \eqref{coherent-pair-intro-b}. For instance, using the same ideas, one may deduce and prove rigorously the following one, which is the situation of \eqref{coherent-pair-intro-b} with $k=0$.

\begin{theorem}\label{pi-N-main-result-general}
Let $({\bf u}, {\bf v})$ be a pair of regular functionals with respect to the pair of monic OPS $\big((P_n)_{n\geq 0},(Q_n)_{n\geq 0}\big)$. Assume that the following equation holds
\begin{align}\label{last-situation-k=0}
\pi_N(x)P_{n}^{[m]} (z) = \sum_{j=n-M}^{N} c_{n,j} Q_{j} (z)\;, \quad c_{n,n-M}\neq 0\;,~~(n=0,1,\ldots)\;,
\end{align}
where $\pi_N$ is a polynomial of degree $N$, $c_{n,j}$, $j=0,1,\ldots, n$, $n=0,1,\ldots$, are complex numbers and $M,m \in \mathbb{N}_0$. Assume that  
$$c_{n,n-M} \neq 0 \quad \quad (n=0,1,\ldots)\;.$$
Let's define 
\begin{align*}
\rho_{M+m+n} (z)=\sum_{j=n-N} ^{n+M} (-1)^m c_{j,n}\frac{\gamma_{m+j}!\left\langle {\bf v},Q_{n} ^2 \right\rangle}{\gamma_j !\left\langle {\bf u},P_{m+j} ^2 \right\rangle}P_{m+j}(z)\quad (n=0,1,\ldots)\;,
\end{align*}
so that $\deg \rho_{M+m+n}=M+m+n$. Assume further that $\det (\mathcal{C}_{m+3}) \neq 0$, where $\mathcal{C}_{m+3}(z)=\Big[e_{i,j}(z)\Big]_{i,j=0} ^{m+2}$ is the following polynomial matrix of order $m+3$
\begin{align*}
e_{i,j}(z) = \left\{
\begin{array}{lclcl}
\displaystyle \texttt{U}_1\mathrm{D}_x\rho_{M+m+i}(z)-\alpha \mathrm{S}_x\rho_{M+m+i}(z), ~ \textit{if}~~j=0  \\ [7pt]
\alpha \mathrm{T}_{m+1,j}(\pi_N Q_i)(z),~\textit{otherwise .}
\end{array}
\right. 
\end{align*}
Then the following hold.
\begin{itemize}
\item[i. ]If $m=0$, then ${\bf u}$ and ${\bf v}$ are connected by a rational modification and so ${\bf u}$ is semiclassical if and only if ${\bf v}$ is so.\\
\item[ii. ]If $m\neq 0$, then  ${\bf u}$ and ${\bf S}_x ^{m}{\bf v}$ are semiclassical functionals and in addition ${\bf S}_x{\bf u}$ and ${\bf S}_x ^{m+1}{\bf v}$ are connected by a rational modification. 
\end{itemize}
\end{theorem}

\begin{remark}
For $q$-quadratic lattices, we emphasize that some particular cases of \eqref{last-situation-k=0} have been considered by several authors. For instance, the case of \eqref{last-situation-k=0} where $(P_n)_{n\geq 0}\equiv (Q_n)_{n\geq 0}$, $N\leq 2$, $M=1$ and $m=1$ is solved in \cite{KDP2022c} as an answer to a conjecture posed by M. E. H. Ismail in \cite{I2005}. Furthermore, the case of \eqref{last-situation-k=0} where $(P_n)_{n\geq 0}\equiv (Q_n)_{n\geq 0}$, $N\leq 2k$ and $M=N$ is treated in a recent work available in \cite{KCDM2023}.
\end{remark}

Despite these results in Theorem \ref{main-result-general-for-m=k} and Theorem \ref{pi-N-main-result-general}, there are still some interesting questions related to the developed theory. Firstly, the situation of \eqref{coherent-pair-intro-b} where $k\neq 0$ therein is still open problem. Secondly, an interesting question is to provide non-trivial examples for \eqref{coherent-pair-intro} and \eqref{coherent-pair-intro-b} considered in this note. As noticed in \cite[p.8]{RJPS2015}, this is not an easy task. Finally, as mentioned at the introduction, this notion of coherent pair of measures has a connection with Sobolev OPS with application in approximation theory. For OPS on lattices, this is an interesting question to be treated in a possible future work.

\section*{Acknowledgements }
The author was partially supported by CMUP, member of LASI, which is financed by national funds through FCT - Fundac\~ao para a Ci\^encia e a Tecnologia, I.P., under the projects with reference UIDB/00144/2020 and UIDP/00144/2020. The author also thank CMUC-UIDB/00324/2020, funded by the Portuguese Government through FCT/MCTES for offering him a one month research visit, period during which this work was initiated under the guidance of Dr. Kenier Castillo.

{

\end{document}
\begin{thebibliography}{99}





\bibitem{RKDP2020} {R. \'Alvarez-Nodarse, K. Castillo, D. Mbouna, and J. Petronilho},
{On discrete coherent pairs of measures}, J. Difference Equ. Appl., vol. {\bf 28}, no. {\bf 7} (2022) 853-868.

\bibitem{RJPS2015} {R. \'Alvarez-Nodarse, J. Petronilho, N.C. Pinz\'on-Cort\'es, R. Sevink-Adig\"uzel }, {On linearly related sequences of difference derivatives of
discrete orthogonal polynomials}, J. Comput. and Appl. Math. {\bf 284} (2015) 26-37.

\bibitem{ARS1995} {N.M. Atakishiev, M. Rahman, and S. K. Suslov},
{On classical orthogonal polynomials}, Constr. Approx. {\bf 11} (1995) 181-226.

\bibitem{BLN1987} {S. Bonan, D. Lubinsky, P. Nevai}, {Orthogonal polynomials and their derivatives II}, SIAM J. Math. Anal. {\bf 18} (1987) 1163–1176.

\bibitem{KCDM2023} {K. Castillo and D. Mbouna}, Proof of two conjectures on Askey-Wilson polynomials, Proc. Amer. Math. Soc., DOI: 10.1090/proc/16250.

\bibitem{KCDM2020} {K. Castillo and D. Mbouna},
{On another extension of coherent pairs of measures}, Indag. Math. (2020) {\bf 31} 223-234.

\bibitem{KDP2020a}
K. Castillo, D. Mbouna, and J. Petronilho,
\textit{On the functional equation for classical orthogonal polynomials on lattices}, J. Math. Anal. Appl. {\bf 515} (2022) 126390.

\bibitem{KDP2022c} {K. Castillo, D. Mbouna, and J. Petronilho}, A characterization of continuous q-Jacobi, Chebyshev of the first kind and Al-Salam Chihara polynomials, J. Math. Anal. Appl. {\bf 514} (2022) 126358.

\bibitem{C1978} {T. S. Chihara}, An introduction to orthogonal polynomials.
Gordon and Breach, New York; 1978.


\bibitem{FK-NM2011} {M. Foupouagnigni, M. Kenfack-Nangho, and S. Mboutngam},
{Characterization theorem of classical orthogonal polynomials
on nonuniform lattices: the functional approach},
Integral Transforms Spec. Funct. {\bf 22} (2011) 739-758.



\bibitem{SMMFPN2017} {S. Mboutngam, M. Foupouagnigni, P. Njionou Sadjang}, {On the modifications of semi-classical orthogonal polynomials on nonuniform lattices}, J. Math. Anal. App. {\bf 445} (2017) 819-836.

\bibitem{KLM2001} {K.H. Kwon, J.H. Lee, F. Marcell\'an}, {Generalized coherent pairs}, J. Math. Anal. Appl. {\bf 253} (2001) 482–514.

\bibitem{MBP1993} {F. Marcell\'an, A. Branquinho, J. Petronilho}, {On inverse problems for orthogonal polynomials. I}, J. Comput. Appl. Math. {\bf 49} (1993) 153–160.

\bibitem{MBP1994} {F. Marcell\'an, A. Branquinho, J. Petronilho}, {Classical orthogonal polynomials: a functional approach}, Acta Appl. Math. {\bf 34 (3)} (1994) 283–303.

\bibitem{MMM2001} F. Marcell\'an, A. Mart\'inez-Finkelshtein, J. Moreno-Balc\'azar, $k$-Coherence of measures with non-classical weights, in: Luis Espanol, Juan L. Varona (Eds.),
Margarita Mathematica en Memoria de Jos\'e Xavier Guadalupe Hern\'andez, Servicio de Publicaciones, Universidad de la Rioja, Logrono, Spain, 2001.


\bibitem{IKNS1991} A. Iserles, P. E. Koch, S. P. N{\o}rsett, and J. M. Sanz-Serna:
{On polynomials orthogonal with respect to certain Sobolev inner products},
J. Approx. Theory \textbf{65} (1991) 151--175.


\bibitem{I2005}
{M. E. H. Ismail},
{Classical and quantum orthogonal polynomials in one variable. With two chapters by W. Van Assche. With a foreword by R. Askey.}, Encyclopedia of Mathematics and its Applications {\bf 98}, Cambridge University Press, Cambridge, 2005.


\bibitem{JMPN2015} M. N. de Jesus, F. Marcell\'an, J. Petronilho, and N.C. Pinz\'on-Cort\'es: {$(M,N)-$coherent pairs of order $(m,k)$ and Sobolev orthogonal polynomials},
J. Comput. Appl. Math. \textbf{256} (2014) 16--35.

\bibitem{JP2008} M. N. de Jesus and J. Petronilho:
{On linearly related sequences of derivatives of orthogonal polynomials},
J. Math. Anal. Appl. \textbf{347} (2008) 482--492.

\bibitem{JP2013} M. N. de Jesus and J. Petronilho:
{Sobolev orthogonal polynomials and $(M,N)-$coherent pairs of measures},
J. Comput. Appl. Math. \textbf{237} (2013) 83--101.



\bibitem{M1991}
{P. Maroni}, Une th\'eorie alg\'ebrique des polyn\^omes orthogonaux. Applications aux polyn\^omes orthogonaux semiclassiques, In C. Brezinski et al. Eds., Orthogonal Polynomials and Their Applications, Proc. Erice 1990, IMACS, Ann. Comp. App. Math. {\bf 9} (1991) 95-130.

\bibitem{MS2000} P. Maroni and R. Sfaxi:
\textit{Diagonal orthogonal polynomial sequences},
Methods Appl. Anal. {\bf 7} (2000) 769--791.

\bibitem{JP2006} 
{J. Petronilho}, On the linear functionals associated to linearly related sequences of orthogonal polynomials, J. Math. Anal. Appl. {\bf 315} (2006) 379-393.


\end{thebibliography}
